\documentclass[11pt,reqno]{amsart}
\usepackage{xifthen}

\usepackage{amsmath}
\usepackage{amssymb, latexsym, amsfonts, amscd, amsthm, mathrsfs}
\usepackage[usenames,dvipsnames]{color}
\usepackage[all]{xy}
\usepackage{graphicx}
\usepackage{epsfig}
\usepackage[colorlinks=true,linkcolor=blue]{hyperref}
\usepackage{amssymb}
\usepackage{enumerate}
\usepackage{bm}

\numberwithin{equation}{section}

\definecolor{purple}{rgb}{0.9,0,0.8}

\definecolor{gray}{rgb}{0.7,0.7,0.7}

\newtheorem{thm}{Theorem}[section]

\newtheorem{lem}[thm]{Lemma}
\newtheorem{prop}[thm]{Proposition}

\theoremstyle{definition}

\newtheorem{rmk}[thm]{Remark}
\newtheorem*{rmk*}{Remark}

\newcommand{\beq}{\begin{equation}}
\newcommand{\eeq}{\end{equation}}

\newcommand{\vep}{\varepsilon}

\newcommand{\sN}{\mathscr{N}}

\newcommand{\ra}{\rightarrow}

\renewcommand{\emptyset}{\varnothing}

\newcommand{\f}{\frac}

\renewcommand{\setminus}{\backslash}

\newcommand{\per}{\mathbf{Per}}
\newcommand{\whp}{{\bf whp}}
\begin{document}

\title{Site percolation on non-regular pseudo-random graphs}

\author[Suman Chakraborty]{Suman Chakraborty$^1$}
\address{$^1$Department of Statistics and Operations Research, 304 Hanes Hall, University of North Carolina, Chapel Hill, NC 27599}
\email{sumanc@live.unc.edu}

\date{}

\maketitle

\begin{abstract}
We study site percolation on a sequence of graphs $\{G_n\}_{n\geq1}$ on $n$ vertices where degree of each vertex is in the interval $(np -a_n, np+a_n)$ and the co-degree of every pair of vertices is at most ${n}p^2+ b_n$, where $p \in (0,1)$ and $\{a_n\}_{n\geq1}$, $\{b_n\}_{n\geq1}$ are sequences of real numbers. Under suitable conditions on $p \in (0,1)$, $a_n$'s and $b_n$'s we show that site percolation on these sequences of graphs undergo a sharp phase transition at $\frac{1}{np}$. More precisely for $\vep>0$, we form a random set $R(\rho_n)$ by including each vertex of $G_n$ independently with probability $\rho_n$. If $\rho_n = \frac{1-\vep}{np}$, then for every small enough $\vep>0$ and $n$ large enough, all connected components in the subgraph of $G_n$ induced by $R(\rho_n)$ are of size at most poly-logarithmic in $n$ with high probability. If $\rho_n = \frac{1+\vep}{np}$, then for every small enough $\vep>0$ and $n$ large enough, the subgraph of $G_n$ induced by $R(\rho_n)$ contains a `giant' connected component of size at least $\frac{\vep }{p}$ with high probability. Further, we show that under an additional assumption on $\{b_n\}_{n\geq 1}$ the giant component is unique.  This partially resolves a question by Krivelevich \cite{krivelevich2016phase} regrading uniqueness of the giant component of site percolation in a general class of regular pseudo-random graphs. We hope that our method of proving uniqueness of the giant component will be applicable in other contexts as well.
\end{abstract}
\section{Introduction}
In the study of reliability of communication networks a question of practical interest \cite{li2015network} is ``how many failed nodes/edges will breakdown the whole network?" {Percolation theory} provides a viable avenue to explore this question (see \cite{li2015network} for a detailed discussion). In {percolation theory} failure of a node/edge is modeled by deletion of that node/edge. See, for example \cite{boccaletti2006complex}, \cite{cohen2010complex}, \cite{pan2016cyclades} for other applications of percolation models. In this article we will focus on the question when each node fails with a fixed probability and independently of all other nodes. This is commonly referred as {vertex percolation} or {site percolation}. Let us now formally state the problem. \par
Consider a sequence of graphs $G_n=(V_n,E_n)$ on $n$ nodes. Fix $0< \rho_n<1$. Form a random subset $R(\rho_n)\subset V_n$ by including each vertex $\nu \in V_n$ independently with probability $\rho_n$. Let $G_n[R(\rho_n)]$ denote the random subgraph of $G_n$ induced by $R(\rho_n)$. In this paper our main objective is to study the size of the maximal component of $G_n[R(\rho_n)]$. We will suppress dependence of $\rho_n$ and $G_n[R(\rho_n)]$ on $n$. Henceforth we will simply write $\rho$ and $G[R(\rho)]$ for simplicity.

\subsection*{Assumptions on the ground graph sequence}
Our main result applies to a particular class of graphs. We now describe the assumptions that we will make on the ground graph sequence $G_n=(V_n,E_n)$. All graphs in this paper are unweighted, undirected and simple. For convenience, we will assume that the vertices are labelled using $[n] := \{1,\ldots,n\}$. Finally, we will work under one or more of the following assumptions. Let $\sN_{\nu}$ be the set of neighbors of the vertex $\nu$ for all $\nu \in [n]$ and $|.|$ be the cardinality. 

\noindent
{\bf Assumption A1.}
$$\min_{v \in [n]}{ \left|\sN_{\nu}  \right|  } > np - a_n.$$

\noindent
{\bf Assumption A2.} 
$$
\max_{v_1\ne v_2 \in [n]}{\left|\sN_{\nu_1} \cap \sN_{\nu_2} \right| }< {n}p^2+ b_n. $$

\noindent
{\bf Assumption A3.}
$$\max_{v \in [n]}{ \left|\sN_{\nu}  \right|  } < np + a_n.$$

\medskip

Here $p$ can (and will) depend on $n$ but we suppress dependence on $n$ for simplicity. Further, the constants $a_n$ and $b_n$ may depend on $p$. \par

Throughout the paper the {``statement $A$ is true \whp (with high probability)"} will mean that {``the probability that the statement $A$ is true goes to one as $n$ (number of vertices) goes to infinity"}. All asymptotics are as $n \ra \infty$. Also we intentionally omit the use of ``floor" or ``ceiling" notations to keep the presentation clearer. \par Recall that $R(\rho)\subset V_n$ is a random subset formed by including each vertex $\nu \in V_n$ with probability $\rho$ independently. Let $G[R(\rho)]$ denote the random subgraph of $G_n$ induced by $R(\rho)$. Our first result (Theorem \ref{thm:1234oct27}) describes the ``supercritical regime". More precisely, it states when $\rho$ is above a certain threshold then $G[R(\rho)]$ contains a large connected component. 
\begin{thm}\label{thm:1234oct27}
Let $\vep>0$. Let $G_n = (V_n, E_n)$ be a sequence of graphs satisfying conditions \textbf{A1} and \textbf{A2} with $p=o(1)$, $np^2  \ra \infty$, $a_n = o(np)$ and $b_n = o(np^2)$.  If $\rho =\f{1+\vep}{np}$, then for every small enough $\vep$ and large enough $n$ the graph $R(\rho)$ contains a connected component of size at least $\f{\vep}{p}$ {\whp}.
\end{thm}
\begin{rmk} 
The assumptions \textbf{A1} and \textbf{A2} with $a_n = o(np)$, $b_n = o(np^2)$ in Theorem \ref{thm:1234oct27} ensure that the graph sequence $G_n$ exhibits ``pseudo-random" properties. Informally, these properties ensure that the sequence of graphs $G_n$ resembles with an {\em Erd\H{o}s-R\'{e}nyi} graph with edge-density $p$. See \cite{alon1999list}, \cite{krivelevich2006pseudo} for different notions and properties of pseudo-random graphs and many examples of such graphs. 

\end{rmk}
In Theorem \ref{thm:1234oct27} we obtained, when $\rho =\f{1+\vep}{np}$ there is a connected component in $G[R(\rho)]$ of size at least $\f{\vep}{p}$ \whp. Note that $\f{\vep}{p}$ is a positive fraction of expected number of nodes in $R(\rho)$. It is natural to ask the size of the second largest component of $R(\rho)$. Our next result states that the second largest component is much smaller than the largest component. In particular it gives uniqueness of the giant component.
\begin{thm}\label{thm:655oct3}
Let $G_n=(V_n, E_n)$ be a sequence of graphs satisfying {\bf A1}, {\bf A2} and {\bf A3}. Also let $p=o(1)$, $np^2  \ra \infty$, $a_n =o(np)$ and $b_n =o(np^3)$. If $\rho =\f{1+\vep}{np}$, then for every small enough $\vep>0$ and large enough $n$, $G[R(\rho)]$ will have a component of size at least $\f{\vep}{p}$ {\whp} and the second largest component will be of size at most $O\left((\log n)^2\right)$ \whp. 
\end{thm}
We do not need to assume {\bf A3} and $b_n =o(np^3)$ in Theorem \ref{thm:655oct3} if we assume a hereditary degree condition. Precisely the assumption is as follows.

\noindent
{\bf Assumption HD.}
For each $\beta>0$ and $n\geq N(\beta)$, every large subgraph $U \subset V_n$, say for $|U|\geq 0.9n$ satisfy,
\[
\max_{\nu \in U}d(\nu, U) < (1+ \beta)p|U|,
\]
where $d(\nu, U)$ denotes the number of neighbors of $\nu$ in the set $U$. The following proposition shows uniqueness of the giant component under different conditions than in Theorem \ref{thm:655oct3}. We believe our technique of proving uniqueness can be applied in other situations as well. 

\begin{prop}\label{prop:dec403thurs}
Let $G_n=(V_n, E_n)$ be a sequence of graphs satisfying {\bf A1}, {\bf A2} and {\bf HD}. Also let $p=o(1)$, $np^2  \ra \infty$, $a_n =o(np)$ and $b_n =o(np^2)$. If $\rho =\f{1+\vep}{np}$, then for every small enough $\vep>0$ and large enough $n$, $G[R(\rho)]$ will have a component of size at least $\f{\vep}{p}$ {\whp} and the second largest component will be of size at most $O\left((\log n)^2\right)$ \whp. 
\end{prop}

 Theorem 1 in \cite{krivelevich2016phase} concerns with the ``subcritical regime". This result is applicable to any graph sequence  $G_n = (V_n,E_n)$ on $n$ vertices with maximum degree less than $d_n$. It states, if $\rho = \f{1-\vep}{d_n}$, then {\whp} all connected components of $G[R(\rho)]$ are of size less than $\f{4}{\vep^2} (\ln n)^2$. Theorem 1 in \cite{krivelevich2016phase} immediately implies the following.
\begin{thm}\label{thm:mon612nov} Let $\vep>0$. Let $G_n = (V_n, E_n)$ be a sequence of graphs satisfying conditions \textbf{A3} with $a_n = o(np)$. If $\rho =\f{1-\vep}{np}$, then for every small enough $\vep$ and large enough $n$, size of all connected components of $G[R(\rho)]$ will be less than $O\left((\log n)^2\right)$ {\whp}. 
\end{thm}
\begin{rmk}
Combining Theorem \ref{thm:mon612nov} and Theorem \ref{thm:1234oct27} we get that if a graph sequence satisfies \textbf{A1}, \textbf{A2}, \textbf{A3} with $p=o(1)$, $np^2\rightarrow \infty$, $a_n = o(np)$, $b_n = o(np^2)$ then site percolation on $G_n$ undergoes a sharp phase transition. Precisely if $\rho =\f{1-\vep}{np}$ then the maximal component in $G[R(\rho)]$ is of poly-logarithmic order and for $\rho =\f{1+\vep}{np}$, the size of the largest component is linear in $|R(\rho)|$. Using Theorem \ref{thm:655oct3} we also have that the giant component is unique as long as $b_n = o(np^3)$.  
\end{rmk}

\section{Discussion and Related Work}

Site percolation was studied for many specific graph sequences. For example, site percolation on generalized cubes was studied in \cite{reidys1997random}. In \cite{sivakoff2014site} the author studied site percolation on Hamming Torus. In \cite{grimmett1998critical} the authors obtained relation between the critical probabilities of bond percolation and site percolation for any connected graph. Site percolation on triangular lattice was studied in \cite{kesten1998almost}. Confidence interval for the critical probabilities for many other Archimedean lattices are given in \cite{riordan2007rigorous}. An upper and lower bound for site percolation on random quadrangulations of the half-plane was obtained in  \cite{bjornberg2015site}. \par In this article we study site percolation on a general class of models satisfying mild pseudo-randomness criteria. Roughly, pseudo-random graphs are sequence of graphs that resemble with a true random graph with appropriate edge-density. The notion of pseudo-random graphs was first introduced by Andrew Thomason \cite{thom1}, \cite{thom2}. Chung, Graham, and Wilson \cite{chung1989quasi} showed many notions of pseudo-randomness are equivalent. The paper by  Krivelevich and Sudakov  \cite{krivelevich2006pseudo} contains an extensive survey of pseudo-random graphs. The notion of pseudo-randomness used in this paper are  similar to the one used in \cite{alon1999list}.\par
Site percolation on $d$ regular pseudo-random graphs was studied in a recent work of Krivelevich \cite{krivelevich2016phase} which is the main inspiration of our work. More precisely \cite{krivelevich2016phase} studied site percolation on $(n,d,\lambda)$ graphs. In $(n,d,\lambda)$ graphs are $d$ regular graphs on $n$ vertices and $\lambda$ is the second largest eigen-value of the adjacency matrix of the graph in absolute value. It was shown in \cite{krivelevich2016phase} under mild assumptions on $\lambda$ these graphs undergo a a sharp phase transition at $\f{1}{d}$. Motivated by applications \cite{li2015network}, we extend this study to a class of non-regular graphs. The class of graphs we have considered in this paper contains a class of $(n,d,\lambda)$ graphs where $d>>\sqrt{n}$. In \cite{krivelevich2016phase}, the author proposed a problem to prove uniqueness of the giant component in super critical regime (when $\rho = \f{1+\vep}{d}$). Theorem \ref{thm:655oct3} not only proves uniqueness of the giant component, it also gives the second largest component must be of poly-logarithmic order. Thus it partially answers the question raised by Krivelevich \cite{krivelevich2016phase} as our result is not applicable to {\bf all} $(n,d,\lambda)$ graphs. More specifically our result is applicable to those $(n,d,\lambda)$ graphs that satisfy the conditions in Theorem \ref{thm:655oct3} or Proposition \ref{prop:dec403thurs}. We are currently investigating how to extend our results in more general setting.

\section{Proof sketch}
In this section we informally discuss the main ideas behind the proofs of our main results and the detailed proofs are deferred to the next section. We used depth first algorithm(DFS) to reveal the connected components of a randomly induced subgraph of the ground graph $G_n$. Then we make use of our assumptions \textbf{A1} and \textbf{A2} to show that all subsets $H\subset V_n$ of appropriate size are {\em expanding}, more formally for $\alpha_0 \in (0,1]$ if $n$ is large enough then $|\sN_G(H)| \geq (1-\alpha_0)\left(npm -\f{np^2 m^2}{2}\right)$ where $\sN_G(H) := \{\nu \in G \setminus H: \nu \text{ has a neighbor in }H\}$. The proof of the fact that sets of appropriate size are expanding is done using an inclusion-exclusion inequality. Finally we use the last fact and DFS to complete the proof in the similar way as in \cite{krivelevich2016phase}. \par
The proof of uniqueness is based on a combinatorial argument. We believe that this method can be adapted to prove uniqueness of the giant component in other settings as as well. We informally sketch our ideas here and the details are done in the subsequent sections. Theorem \ref{thm:1234oct27} gives us that there is a component of size at least $\f{\vep}{p}$. For any subgraph $S$ of $G_n$ let $\mathcal{O}_{G_n}(S)$ denote the set of nodes that are not immediate neighbor to $S$. Let $C$ be a connected subgraph of $G_n$ of size $\f{\vep}{p}$. At first we show that $\mathcal{O}_{G_n}(C)$ is of size approximately at most $(1-\vep)n$. In the second step we show that the maximum degree a vertex in the subgraph of $G_n$ induced by $\mathcal{O}_{G_n}(C)$ is bounded above by $(1+\vep^5)(1-\vep)np$ for all but few `bad' vertices. Finally we show that probability that at least one of the `bad' vertices is getting selected is small and use Theorem 1 from \cite{krivelevich2016phase} on the subgraph of $G_n$ induced by $\mathcal{O}_{G_n}(C)$ to obtain the result.

\section{Notations and Preparatory Lemmas}

We summarize the notations that we will use in the proof. For a graph $G=(V, E)$, $\sN_G(H) := \{\nu \in G \setminus H: \nu \text{ has a neighbor in }H\}$, $\mathcal{O}_{G_n}(C) := $. For $U\subset V$, we will write $\per_{G,\rho}(U)$ to denote the induced subgraph of $G$ by the set formed by including each vertex of $U$ independently with probability $\rho$. Note that $\per_{G_n,\rho}(V_n) = G[R(\rho)]$. For $\nu \in V$ and $U \subset V$, $d(\nu, U)$ denotes the number of neighbors of $\nu$ in the set $U$. Also $(a_{\nu_1 \nu_2})_{\nu_1,\nu_2=1}^n$ will denote the adjacency matrix of the graph $G$. 
\subsection{Depth First Search Algorithm(DFS)}
This is a widely used algorithm to find out the connected components of a graph. We will use DFS to reveal connected components of a randomly induced subgraph $G[R(\rho)] $ of the graph $G_n=(V_n, E_n)$ in the same way as in  \cite{krivelevich2016phase}. We state it here for completeness. At any particular instance it partitions the set into four sets. $\mathcal{S}$ is the set of vertices whose exploration is complete. $\mathcal{T}$ is the set of vertices that are yet to be visited. $\mathcal{U}$ is the set of vertices that are kept in the stack (last in first out) and $\mathcal{W}$ is the set of vertices that are found to fall outside $R(\rho)$. The algorithm proceeds as follows.  

\begin{itemize}
\item Starts with $\mathcal{S} = \mathcal{U} = \mathcal{W} = \emptyset$ and $\mathcal{T} = V$. 
\item If \textbf{$\mathcal{U}$ is empty} then it selects the first vertex in $\mathcal{T}$ according to the natural ordering in $[n]$, deletes it from $\mathcal{T}$ and with probability $\rho$ it is put in $\mathcal{U}$ otherwise put it in $\mathcal{W}$.
\item If \textbf{$\mathcal{U}$ is not-empty} then the algorithm queries $\mathcal{T}$ for neighbors of the last vertex $\nu$ that was inserted in $\mathcal{U}$ according to the natural order in $[n]$. If it has a neighbor $\nu'$ in $\mathcal{T}$ then it gets deleted from $\mathcal{T}$ and added to $\mathcal{U}$ with probability $\rho$ otherwise $\nu'$ is moved to $\mathcal{W}$. If $\nu$ does not have a neighbor in $\mathcal{T}$ then it is moved to $\mathcal{S}$.
\item The algorithm ends when $\mathcal{U} \cup \mathcal{T}$ is empty. At this point $\mathcal{S} = R(\rho)$ and $\mathcal{S} = V_n \setminus R(\rho)$. 
\end{itemize}

\begin{rmk}
Observe that a connected component starts to get revealed when for the first time a vertex from that component appears in $\mathcal{U}$, which was empty and completely reveals the connected component when $\mathcal{U}$ becomes empty again. Following \cite{krivelevich2016phase}, we will call the time between two consecutive emptying of $\mathcal{U}$, an {\em epoch}. Also at any time point in the DFS algorithm $\mathcal{N}_G(S) \subset \mathcal{U} \cup \mathcal{W}$. Finally, notice that at the end of the algorithm we will get all the connected component of $R(\rho)$ when at each stage the DFS algorithm is fed with i.i.d Bernoulli($\rho$) random variables.  Denote the sequence by $\tilde{Y} = (Y_i)_{i=1}^n$.
\end{rmk}

\subsection{Technical Lemmas}
The following two Lemmas are the main ingredients for proof (Theorem \ref{thm:1234oct27}) of existence of a giant component in ``supercritical regime".
\begin{lem}\label{lem:monoct900}
Let $G_n = (V_n,E_n)$ be a graph sequence satisfying {\bf A1} and {\bf A2}. Let $0< c <1/3$ be a constant. Suppose that $np^2 \rightarrow \infty$ and $a_n = o(np)$, $b_n =o(np^2)$. Then there is no set $H$ with $|H|=m$, $ c<mp\leq \f{1}{3}$  that satisfies $|\sN_G(H)|< (1-\alpha_0)\left(npm -\f{np^2 m^2}{2}\right)$ with $0< \alpha_0\leq 1$ when $n$ is large enough. 
\end{lem}
\begin{proof}[Proof of Lemma \ref{lem:monoct900}]
Let $H$ be a set with $|H|=m$. We have
\[ 
|\sN_G(H)| = |\cup_{\nu \in H}{\sN_{\nu}}\setminus H| \geq |\cup_{\nu \in H}{\sN_{\nu}}| - |H|.
\]
Now using inclusion-exclusion  we have,
\begin{equation}\label{eqn:302:oct28}
|\sN_G(H)| \geq \sum_{\nu \in H}{|\sN_{\nu}|} - \sum_{ \nu <\nu' \in H}{|\sN_{\nu} \cap \sN_{\nu'}|} -|H|.
\end{equation}
Plugging in {\bf A1} and {\bf A2} we get,
\begin{equation}
|\sN_G(H)| \geq |H| \left[ (np - a_n) - \f{\left(|H| - 1 \right)}{2} (np^2 + b_n) - 1\right].
\end{equation}
Since { $b_n = o(np^2)$, $a_n=o(np)$, $np^2 \rightarrow \infty$ and $c<m \leq \f{1}{3p}$} we have for every $0 \leq \alpha_0 <1$, there is a positive integer $N(\alpha_0)$ such that for $n\geq N(\alpha_0)$  ,
\begin{equation*}
|\sN_G(H)| \geq  (1-\alpha_0)\left( mnp  - \f{nm^2 p^2}{2} \right).
\end{equation*}
\end{proof}
The next Lemma gives tail probabilities of Binomial distribution. The proofs can be done using Chernoff bound. We refer the reader to \cite{krivelevich2016phase} for proof of the statements. 
\begin{lem}[{\cite[Lemma 2.3]{krivelevich2016phase}}]
\label{lem:monoct835} Let $\vep >0$ be a constant. Then consider a sequence $\tilde{Y} = (Y_i)_{i=1}^n$ of i.i.d. Bernoulli random variables with parameter $\rho$. Assume . Let $\rho = \f{1+\vep}{np}$, if $p = o(1)$ then the following are true for small enough $\vep$ {\bf whp}.
\begin{enumerate}
\item \label{lem:monoct8351}  $\sum_{i=1}^{\vep^3 n }Y_i \leq \f{2\vep^3 }{p}$.
\item \label{lem:monoct8352} $\sum_{i=1}^{\vep n }Y_i \leq \f{2\vep }{p}$.
\item \label{lem:monoct8353} For every $\vep^3 n \leq t\leq \vep n$, $\sum_{i= 1}^{t }Y_i \geq \f{(1+ \f{3\vep}{4})t}{np}$. 
\end{enumerate}
\end{lem}
We will use the following two Lemmas to prove the uniqueness of the giant component. The results might be of independent interest. The next two Lemmas give us the ``correct" lower bound of $d(\nu, U)$, for a fixed subset $U \subset V_n$ for most of the nodes $\nu \in [n]$ when the sequence of graphs satisfies {\bf A1}, {\bf A2}, {\bf A3}.
\begin{lem}\label{lem:thurs514}
Let $G_n= (V_n, E_n)$ be a sequence of graphs satisfying {\bf A1}, {\bf A2}, {\bf A3}. Let $X_n$ be a uniformly distributed random variable on $[n]$, then 
\begin{equation}\label{eqn:thurs514}
\mathrm{Var}(d(X_n, U)) \leq  p(1-p) |U| + \f{|U|(a_n - b_n)}{n}+ \frac{b_n}{n} |U|^2+ 2\frac{a_np}{n} |U| -\f{a_n^2|U|^2}{n^2}.
\end{equation}
\end{lem}
\begin{proof}[Proof of Lemma \ref{eqn:thurs514}]
\begin{align}
\sum_{\nu_1,\nu_2 \in [n]} \{d(\nu_1, U) - d(\nu_2, U)\}^2 &= \sum_{\nu_1,\nu_2 \in [n]} d^2(\nu_1, U) + \sum_{\nu_1,\nu_2 \in [n]} d^2(\nu_1, U)  - 
2\sum_{\nu_1,\nu_2 \in [n]} d(\nu_1, U) d(\nu_2, U) \notag \\
&= n\sum_{\nu_1 \in [n]} d^2(\nu_1, U) + n\sum_{\nu_2 \in [n]} d^2(\nu_2, U) - 2\left( \sum_{\nu \in [n]} d(\nu, U)\right)^2. \label{eqn:oct748}
\end{align}
Firstly using assumption {\bf A1},
\begin{equation}\label{eqn:oct846}
 \sum_{\nu \in [n]} d(\nu, U) = \sum_{\nu \in U}  \left|\sN_{\nu}  \right| > {n}p|U|- a_n|U|,
 \end{equation}
 and
 \begin{align*}
 \sum_{\nu \in [n]} d^2(\nu, U) &= \sum_{\nu \in [n]} \left( \sum_{\rho \in U}  {a_{\nu,\rho}} \right)^2 = \sum_{\nu \in [n]} \sum_{\rho \in U}  {a_{\nu,\rho}^2} +  \sum_{\nu \in [n]}\sum_{\rho \neq \rho' \in U}  {a_{\nu,\rho} a_{\nu,\rho'}} \\
 &=  \sum_{\rho \in U} {\left|\sN_{\rho}  \right|} + \sum_{\rho \neq \rho' \in U} {\left|\sN_{\rho} \cap \sN_{\rho'} \right|}
 \end{align*}
 Now using {\bf A2} and {\bf A3} we get,
 \begin{equation}\label{eqn:oct847}
  \sum_{\nu \in [n]} d^2(\nu, U) \leq |U| (np +a_n) + |U|(|U| - 1) \left({n}p^2+ b_n \right)
  \end{equation}
  Now combining \eqref{eqn:oct748}, \eqref{eqn:oct846}, \eqref{eqn:oct847} we get
\begin{align}
&\sum_{\nu_1,\nu_2 \in [n]} \{d(\nu_1, U) - d(\nu_2, U)\}^2 \leq 2n \left(|U| np + |U| a_n + (|U|^2 -|U|)\left({n}p^2+ b_n \right) \right) - 2\left({n}p|U|- a_n|U| \right)^2  \notag \\
&\leq 2n^2p |U| + 2n|U|a_n+ 2(n p)^2|U|^2 + 2nb_n |U|^2 - 2|U|n^2p^2 - 2\left((np)^2 |U|^2 - 2na_np|U| + a_n^2|U|^2\right) \notag \\
&\leq 2n^2 p(1-p)|U| +2n|U|(a_n-b_n)+ 2nb_n |U|^2 + 4na_np |U|-2a_n^2 |U|^2.
\end{align}
Thus variance of $d(X_n, U)$, where $X_n$ is an uniformly distributed random variable on $[n]$ satisfies the following upper bound,
\[
\mathrm{Var}(d(X_n, U)) \leq p(1-p) |U| + \f{|U|(a_n - b_n)}{n}+ \frac{b_n}{n} |U|^2+ 2\frac{a_np}{n} |U|- \f{a_n^2|U|^2}{n^2}.
\]
\end{proof}

\begin{rmk}
If $|U|=\Theta(n)$ then for large enough $n$ we have $\mathrm{Var}(d(X_n, U)) \leq 2p |U| + \frac{3b_n}{n} |U|^2$.
\end{rmk}
We state the key lemma that we will use in our proof of uniqueness of the giant component.

\begin{lem}\label{lem:thurs512}
Let $G_n=(V_n, E_n)$ satisfy {\bf A1}, {\bf A2} and {\bf A3}, Let $U \subset V$ and $\Xi = \{\nu \in [n]: d(\nu, U) \geq (1+\alpha)p|U|\}$. If $ a_n =o(np)$ and $|U|\geq \f{n}{2}$, then for $n$ large enough,
\begin{equation}\label{eqn:thurs510}
|\Xi| \leq   \f{4}{(\alpha p)^2 } \left(4p+12b_n\right).
\end{equation}
\end{lem}
\begin{proof}[Proof of Lemma \ref{eqn:thurs510}]
First note that using a similar argument as in \eqref{eqn:oct846} and using {\bf A3} we have $\mathrm{E}(d(X_n, U)) \leq p|U| + \f{a_n}{n}|U|$, thus 
\begin{equation}\label{eqn:oct954}
\mathrm{P}\left(d(X_n, U)\geq(1+ \alpha)|U|p\right) \leq \mathrm{P}\left(d(X_n, U) \geq \mathrm{E}(d(X_n, U)) - \f{a_n}{n}|U| +\alpha|U|p \right)
\end{equation}
Now since we have for large $n$, {$ a_n =o(np) $}, set $\vep_n = |U|(\alpha p - \frac{a_n}{n})$.
Using \eqref{eqn:oct954} and Markov's s inequality we get
\begin{align*}
\mathrm{P}\left(d(X_n, U)\geq (1+\alpha)|U|p\right) &\leq \mathrm{P}\left(d(X_n, U) - \mathrm{E}(d(X_n, U)) \geq |U|(\alpha p - \frac{a_n}{n}) \right) \\
& = \mathrm{P}\left(d(X_n, U) - \mathrm{E}(d(X_n, U)) \geq \vep_n \right) \\
&\leq \frac{\mathrm{Var}(d(X_n, U))}{\vep_n^2}
\end{align*}
We have for large $n$, $\vep_n = |U|\left(\alpha p - \frac{a_n}{n}\right) >{|U|}\left(\f{\alpha p}{2}\right)$ and $\mathrm{Var}(d(X_n, U)) \leq 2p|U| + \frac{3b_n}{n} |U|^2$, hence 
\[
\mathrm{P}\left(d(X_n, U) \geq (1+\alpha)|U|p\right) \leq \frac{2p|U| + \frac{3b_n}{n} |U|^2}{{|U|^2}\left(\f{\alpha p}{2}\right)^{2}} \leq \f{1}{\left(\f{\alpha p}{2}\right)^{2}}\f{2p}{|U|} + \f{1}{\left(\f{\alpha p}{2}\right)^{2}}  \frac{3b_n}{n}.
\]
In the last display plugging in $|U|\geq \f{n}{2}$ we get,
\[
\mathrm{P}\left(d(X_n, U)\geq(1+\alpha)|U|p\right) \leq \f{4}{(\alpha p)^2 n} (4p+12b_n).
\]
This gives $|\Xi| \leq   \f{4}{(\alpha p)^2 } \left(4p+12b_n\right)$.
\end{proof}

The next Lemma will provide us a crucial estimate of the number of vertices in the ground graph that are not neighbor to the giant component. 
\begin{lem}\label{lem:154oct27}
Suppose $C$ be a connected subgraph of $G_n$ of size equal to $\f{\vep}{p}$. Then number of vertices in $V\setminus C$ that are not neighbor to $C$(denote it by $\mathcal{O}_{G_n}(C)$) is at most $n(1-\vep +\vep^2 + \vep l_n)$, where $l_n$ is a sequence going to $0$ as $n \ra \infty$.
\end{lem}
\begin{proof}[Proof of Lemma \ref{lem:154oct27}]
Let us compute $\mathcal{N}_G(C)$. First note that since $C$ is connected we have
\[
\mathcal{N}_G(C) = \left|\cup_{\nu \in C} \mathcal{N}_{\nu}\setminus C\right| = \left|\cup_{\nu \in C} \mathcal{N}_{\nu}\right| - \left|C \right|.
\]

We will use inclusion-exclusion to get the following lower bound,
\begin{align}\label{eqn:136oct27}
\mathcal{N}_G(C) &\geq \sum_{\nu \in C} |\mathcal{N}_{\nu}| - \sum_{\nu < \nu' \in C} |\mathcal{N}_{\nu} \cap \mathcal{N}_{\nu'}| -|C| \notag \\
& \geq \f{\vep}{p}\left(np - a_n\right) - \f{\vep^2}{2p^2}\left(np^2 +b_n \right) -|C| \notag \\
& = \vep n \left[1 - \f{a_n}{n} -\f{\vep}{2}(1 + \f{b_n}{np^2}) \right] -|C| \notag \\
&= \vep n\left[1 - \f{\vep}{2} + l_n \right] -|C| 
\end{align}
where $l_n = o(1)$. Thus number of vertices in $G$ that are not neighbor to $C$ and nor in $C$ is $n - \mathcal{N}_G(C) - |C|$ and by \eqref{eqn:136oct27} this is at most $n(1-\vep +\f{\vep^2}{2} + \vep l_n)$.
\end{proof}

\section{Proofs of main Theorems}
This section contains proofs of all the theorems. First we will prove the existence of a giant component in supercritical regime.
\subsection{Proof of Theorem \ref{thm:1234oct27}}
\begin{proof}
Fix $\alpha_0= \vep/10 $, $c= \vep^3$, now for large enough $n$ the conclusion of the Lemma \ref{lem:monoct900} holds. Now we run the DFS algorithm with a sequence of i.id Bernouli($\rho$) random variables. We consider the situation after $\vep n$ many vertex queries (a vertex will be included in $\mathcal{U}$ or not type queries) of the algorithm. Assume at some time point $t \in [\vep^3n, \vep n]$, the set $\mathcal{U}$ becomes empty. Then we must have $|\mathcal{S} \cup \mathcal{W}| = t$ and $|\mathcal{S}| = \sum_{i=1}^tY_i$.  Then by \ref{lem:monoct8352} and \ref{lem:monoct8353} in Lemma \ref{lem:monoct835} {\bf whp} \[ \f{(1+ \f{3\vep}{4})t}{n} \leq p|\mathcal{S}|\leq {2\vep } \leq 1/3\] for small enough $\vep$. Now since at that point $\mathcal{U}$ is empty, $\sN_G(\mathcal{S}) \subset \mathcal{W}$. The function $g(x) :=  x -  \f{x^2}{2n}$, is non-decreasing when $x\leq {n}$, thus it is non-decreasing at $x = np|\mathcal{S}| \leq n/3 < n$. Now since we have ${(1+ \f{3\vep}{4})t} \leq np|\mathcal{S}|$, hence by Lemma \ref{lem:monoct900} we have {\bf whp},
\begin{align}
|\mathcal{W}| &\geq (1-\alpha_0)\left(np|\mathcal{S}| -\f{np^2 |\mathcal{S}|^2}{2}\right) \notag \\
&=(1-\alpha_0)\left(np|\mathcal{S}| -\f{(np |\mathcal{S}|)^2}{2n}\right) \notag \\
&\geq (1- \f{\vep}{10}) (1+ \f{3\vep}{4})t\left(1- \f{1}{2n} (1+ \f{3\vep}{4})t\right) \notag \\
&\geq (1- \f{\vep}{10}) (1+ \f{3\vep}{4})t\left(1- \f{1}{6} (1+ \f{3\vep}{4}){\vep }\right) \notag \\
&>t,
\end{align}
for small enough $\vep$, contradicting our theorem assumption. Thus {\bf whp} all the vertices that are being explored in the time frame $[\vep^3n, \vep n]$ belong to the same epoch and hence the same component. Again using parts \ref{lem:monoct8352} and \ref{lem:monoct8353} of the Lemma \ref{lem:monoct835} we get the size of this component is bounded below by $\f{\vep}{p}$.
\end{proof}

\subsection{Proof of uniqueness under hereditary degree assumption}

First we will prove uniqueness of the giant component under an additional assumption that we will call hereditary degree assumption. It is interesting to note that if in addition to {\bf A1} and {\bf A2} we suppose that the following hereditary property ({\bf HD}) holds for the graph sequence $G_n = (V_n, E_n)$, then the giant component will be unique when $p =o(1)$, $np^2 \ra \infty$, $a_n=o(np)$, $b_n=np^2$. In particular we will not require {\bf A3} and $b_n = o(np^3)$. 

\noindent
{\bf Assumption HD.}
For each $\beta>0$ and $n\geq N(\beta)$, every large subgraph $U \subset V_n$, say for $|U|\geq 0.9n$ satisfy,
\[
\max_{\nu \in U}d(\nu, U) < (1+ \beta)p|U|.
\]

\begin{prop}\label{prop:554oct30}
In addition to the conditions in Theorem \ref{thm:1234oct27} assume that $G$ satisfy {\bf HD}. Then there is an unique giant component with size greater than or equal to $\f{\vep}{p}$ {\bf whp} and all other components are of size less than $O((\ln n)^2)$.
\end{prop}
\begin{proof}[Proof of Proposition \ref{prop:554oct30}]
Let $C(\rho)$ be a component with size at least $\f{\vep}{p}$. Recall that $\mathcal{O}_{G_n}(C(\rho)) := \{\nu \in V\setminus C(\rho): \nu \text{ is {\bf not} a neighbor of } C(\rho) \}$. Lemma \ref{eqn:136oct27} gives $\left|\mathcal{O}_G(C(\rho))\right| \leq n(1-\vep +\f{\vep^2}{2} + \vep l_n)$. At this end, note that all vertices of a component that is not connected to $C(\rho)$ must belong to $\mathcal{{O}}_G(C(\rho))$. Also the largest connected component of $\per_{G_n,\rho}(\mathcal{O}_G(C(\rho)))$ is no more than the largest connected component in a set $\per_{G_n,\rho}(\mathcal{{P}})$ where $\mathcal{P}$ is any set with size  $n(1-\vep +\f{\vep^2}{2} + \vep l_n)$ containing $\mathcal{O}_{G_n}(C(\rho))$. Choose $\vep$ small enough and $n$ large enough so that $n(1-\vep +\f{\vep^2}{2} + \vep l_n) \geq 0.9n$. Now by {\bf HD} $\max_{\nu \in P}d(\nu, \mathcal{P}) < (1+\vep^5)p|\mathcal{P}|$ when $n$ is large. We get that $\mathcal{P}$ is a graph on  $n(1-\vep +\f{\vep^2}{2} + \vep l_n)$ vertices with maximum degree $(1+\vep^5)p|\mathcal{P}| $ and each vertex is retained with probability $\f{1+\vep}{np}$. It is easy to check for $\vep>0$ small enough and $n$ large enough,
$$
\f{1-\vep^2/4}{np(1+\vep^5)(1-\vep +\f{\vep^2}{2} + \vep l_n)} \geq \f{1+\vep}{np}.
$$

Thus the subgraph induced by $G_n$ on $\mathcal{P}$ is a graph on $n(1-\vep +\f{\vep^2}{2} + \vep l_n)$ vertices with maximum degree $(1+\vep^5)p|\mathcal{P}|$ and each of the vertices is retained with probability $\rho$ that is less than ${\left(1-\f{\vep^2}{4}\right)}/{p(1+\vep^5) |\mathcal{P}|}$, hence we appeal directly to Theorem 1 in \cite{krivelevich2016phase} and get that the largest connected  component in $\per_{G,\rho}(\mathcal{P})$ is less than $O\left((\ln n)^2\right)$. 
\end{proof}
\subsection{Proving uniqueness under {\bf A1}, {\bf A2}, {\bf A3}}
Notice that in the proof of Proposition \ref{prop:554oct30}, we only needed $\max_{\nu \in U}d(\nu, U) < (1+ \beta)p|U|$ for a particular set $U$, namely for $U = \mathcal{O}_G(C(\rho))$. From Lemma \ref{lem:thurs512} we have that if $G_n= (V_n, E_n)$ satisfies {\bf A1}, {\bf A2} and {\bf A3} then for any fixed large set $U \subset V_n$, there are not too many vertices in $G_n$ that do not satisfy {\bf HD}.  
\begin{proof}[Proof of Theorem \ref{thm:655oct3}]
The proof is similar to Proposition \ref{prop:554oct30}, except that now we do not have 
\begin{equation}\label{eqn:604oct30}
\max_{\nu \in \mathcal{P}}d(\nu, \mathcal{P}) < (1+\vep^5)p|\mathcal{P}|,
\end{equation}
when $n$ is large. But since $|\mathcal{P}|>n/2$ by Lemma \ref{lem:thurs512}, we have the number of elements in $G$ that do not satisfy \ref{eqn:604oct30} is at most $\f{4}{(\alpha p)^2 } \left(4p+12b_n\right)$, with $\alpha = \vep^5$. Thus it is sufficient to show that the probability that at least one vertex is getting selected out of $\f{4}{(\alpha p)^2 } \left(4p+12b_n\right)$ is going to zero. We show that a Binomial distribution with parameter $\f{4}{(\alpha p)^2 } \left(4p+12b_n\right)$ and $\rho = \f{1+\vep}{np}$ takes the value zero with probability going to one. Indeed, the probability is equal to
\begin{align*}
\left(1 - \f{1+\vep}{np}\right)^{\f{4}{(\alpha p)^2 } \left(4p+12b_n\right)} &= \exp{\left[\f{4}{(\alpha p)^2 } \left(4p+12b_n\right) \ln {\left(1 - \f{1+\vep}{np}\right)}\right]} \\
&\geq \exp{\left[\f{-32(1+\vep)}{\alpha n p^2} - \f{96b_n}{np^3}\right]} \rightarrow 1.
 \end{align*}
 In the second step we used the fact that $\ln(1-x)\geq -2x$ for $x\in(0,\f{1}{2})$. Now proceeding as Proposition \ref{prop:554oct30} we have the proof. 
\end{proof}

\section*{Acknowledgements}
I am particularly indebted to Shankar Bhamidi and Sayan Banerjee for numerous insightful suggestions regarding the contents and organization of the article, their support and encouragement were crucial throughout this work. I would also like to thank UNC Probability Group where I presented a preliminary version of this work and received valuable feedbacks. SC has been partially supported by NSF-DMS grants 160683, 161307 and ARO grant W911NF1710010. 
\bibliographystyle{plain}
\bibliography{siteref.bib}

\end{document}